\documentclass[12pt,reqno]{amsart}
\usepackage{amsaddr}
\usepackage{amssymb,latexsym,amsmath,epsfig,amsthm,mathrsfs}
\usepackage{rotating}
\usepackage{graphicx}
\usepackage{amssymb}
\usepackage{lineno}
\usepackage{enumitem}
\usepackage{cite}
\usepackage[top=3cm, bottom=3cm, left=3cm, right=3cm]{geometry}
\usepackage[usenames]{color}
\usepackage[colorlinks=true,
linkcolor=blue,
filecolor=blue,
citecolor=blue]{hyperref}

\usepackage{mathtools}

\makeatletter
\DeclarePairedDelimiterX{\pmodx}[1]{(}{)}{{\operator@font mod}\mkern6mu#1}
\renewcommand{\pmod}{%
  \allowbreak
  \if@display\mkern18mu\else\mkern8mu\fi
  \pmodx
}
\makeatother

 \newtheorem{theorem}{Theorem}[section]
 
 \newtheorem{lemma}[theorem]{Lemma}
 \theoremstyle{definition}
 
 \theoremstyle{remark}
 
 \theoremstyle{theorem}
 \newtheorem{problem}{Problem}[section]
 \theoremstyle{definition}

 \numberwithin{equation}{section}

\begin{document}

\title[The maximum size of sumsets in finite cyclic groups]{The maximum size of sumsets in finite cyclic groups}

%%%%%%    Information for first author

\author[V Goswami]{Vivekanand Goswami}
\address{\em{\small Department of Mathematics, Indian Institute of Technology Bhilai, Durg – 491001, Chhattisgarh, India\\
email: vivekanandg@iitbhilai.ac.in}}
%\email{vivekanandg@iitbhilai.ac.in}

%%%%%%   Information for second author

\author[R K Mistri]{Raj Kumar Mistri$^{*}$}
\address{\em{\small Department of Mathematics, Indian Institute of Technology Bhilai, Durg – 491001, Chhattisgarh, India\\
email: rkmistri@iitbhilai.ac.in}}

\thanks{$^{*}$Corresponding author}

%    General info


\begin{abstract}
Let $A$ be a nonempty finite subset of an additive abelian group $G$. Given a nonnegative integer $h$, the $h$-fold sumset $hA$ is the set of all sums of $h$ elements of $A$, and the restricted $h$-fold sumset $h^\wedge A$ is the set of all sums of $h$ distinct elements of $A$. The union of restricted sumsets $s^\wedge A$, where $s=0, 1, \ldots, h$, is denoted by $[0, h]^\wedge A$. For fixed positive integers $m$ and $h$, the maximum size of the sumset $hA$ of  a set $A \subseteq G$ with $m$ elements is denoted by $\nu(G, m, h)$. In other words, $\nu(G, m, h) = \max\{|hA| : A \subseteq G, |A|= m\}$. Analogous quantities can be defined for the sumsets $h^\wedge A$ and $[0, h]^\wedge A$. Optimal upper bounds are known for these quantities. If $G$ is a finite cyclic group of order $n$, then each of these quantities agrees with the optimal upper bound, except in many cases. Bajnok posed the problem of determining all positive integers $n$, $m$, and $h$ for which the value of the function $f(n, m, h)$ is strictly less than the optimal upper bound. He posed similar problems for quantities related to the sumsets $h^\wedge A$ and $[0, h]^\wedge A$. We prove that, for any positive integer $h$, there are infinitely many positive integers $m$ and $n$ such that $\nu(\mathbb{Z}_n, m, h)$ is strictly less than the optimal upper bound. We also prove similar results for quantities related to the sumsets $h^\wedge A$ and $[0, h]^\wedge A$ also. These results provide the partial solutions to the problems posed by Bajnok.		
\end{abstract}

%%% ----------------------------------------------------------------------
\maketitle
%%% ----------------------------------------------------------------------

%\tableofcontents

\section{Introduction}
Throughout the paper, let $\mathbb{Z}$ denote the set of integers. Let $\mathbb{N}$ denote the set of  positive integers, and  let $\mathbb{N}_0$ denote the set of  nonnegative integers. Let $a$ and $b$ be two integers such that $a \leq b$, then the set $\{n \in \mathbb{Z}: a \leq n \leq b\}$ is denoted by $[a, b]$. The cardinality of a set $A$ is denoted by  $|A|$. If $A$ is a subset of an additive abelian group $G$, and $g \in G$, then we define $A + g = \{a + g: a \in A\}$. For a positive integer $n$, let $\nu_2(n)$ denote the largest nonnegative integer $r$ such that $2^r$ divides $n$. For a real number $x$, let $\lfloor x \rfloor$ denote the greatest integer less than or equal to $x$. For an integer $n \in \mathbb{N}$, let $\mathbb{Z}_n$ denote the additive cyclic group of order $n$. 
	
Let $A = \{a_1, \dots, a_{m}\}$ be a nonempty finite subset of an additive abelian group $G$. Let $\Lambda\subseteq\mathbb{Z}$ and $H\subseteq\mathbb{N}_0$. Following the notations in \cite{bajnok2018}, the sumset of $A$ corresponding to $\Lambda$ and $H$, denoted by $H_{\Lambda} A$, is defined  as
\begin{equation}
	\begin{aligned}\label{h-lambda-sumset}
		H_{\Lambda} A = \{ \lambda_1 a_1 + \cdots + \lambda_m a_m : (\lambda_1, \dots, \lambda_m ) \in \Lambda^m (H)\},
	\end{aligned}
\end{equation}
where the index set $\Lambda^m (H)$ is defined as
\[\Lambda^m(H)=\{( \lambda_1,\dots,\lambda_m )\in \Lambda^m: |\lambda_1| + \cdots + |\lambda_m| \in H \}.\]
Corresponding to the sumset  $H_{\Lambda}A$, the function $\nu_\Lambda(G, m, H)$ is defined  as
\[\nu_\Lambda(G, m, H) = \max\{|H_{\Lambda}A|: A \subseteq G, |A|=m \}.\] 
It can be shown that if $G$ is a finite abelian group, then 
\begin{equation*}
  \nu_\Lambda(G, m, H) \leq \min \{|G|, |\Lambda^m(H)|\}.
\end{equation*}

Let $h$ be a nonnegative integer. We are interested in the following special cases of the sumset of $A$ defined in \eqref{h-lambda-sumset} and the corresponding function $\nu_\Lambda(G, m, H)$: 	
\begin{itemize}
  \item The {\it $h$-fold sumset $hA$} which correspond to $H = \{h\}$ and $\Lambda = [0, h]$. That is,
        \begin{equation}\label{h-fold-sumset}
           hA = \left\{\sum_{i=1}^{m}\lambda_i a_i : \lambda_1,\ldots, \lambda_m \in [0, h], \ \sum_{i = 1}^{m}\lambda_i = h\right\}.
        \end{equation}
  \item The {\it restricted $h$-fold sumset $h^\wedge A$} which correspond to $H = \{h\}$ and $\Lambda = [0, 1]$. That is,
        \begin{equation}\label{res-h-fold-sumset}
           h^\wedge A = \left\{\sum_{i=1}^{m}\lambda_i a_i : \lambda_1, \ldots, \lambda_m \in [0, 1], \ \sum_{i = 1}^{m}\lambda_i = h\right\}.
        \end{equation}
  \item The restricted $[0, h]$-fold sumset $[0 , h]^\wedge A$ which correspond to $H = [0, h]$ and $\Lambda = [0, 1]$. That is,
        \begin{equation}\label{res-sumsets-union}
           [0, h]^\wedge A = \left\{\sum_{i=1}^{m}\lambda_i a_i : \lambda_1, \ldots, \lambda_m \in [0, 1], \ \sum_{i = 1}^{m}\lambda_i \in [0,h]\right\}.
        \end{equation}
\end{itemize}

The study of the sumsets has a long history with {\em Cauchy-Davenport theorem}, which is due to Cauchy \cite{cauchy, dav1, dav2} and is the earliest known result in the field of additive combinatorics which states that if $A$ is a nonempty subsets of $\Bbb Z_p$, then $|hA| \geq \min(p, h|A| - h + 1)$ for all positive integers $h$. The corresponding theorem for the restricted $h$-fold sumset $h^{\wedge}A$ in $\Bbb Z_p$ is due to Dias da Silva and Hamidoune (see \cite{dias}; see also \cite{alon1, alon2}) which states that if $A$ is a nonempty subset of $\Bbb Z_p$, then $|h^{\wedge}A| \geq \min(p, h|A| - h^2 + 1)$ for all positive integers $h \leq |A|$. These sumsets, as well as other kind of sumsets, have been studied extensively in the litearture (see \cite{nath, tao, freiman, mann} and the references given therein). 

One of the important problems in this field is estimating the maximum possible size of sumsets. Corresponding to the sumsets \eqref{h-fold-sumset}, \eqref{res-h-fold-sumset} and \eqref{res-sumsets-union}, the function $\nu_\Lambda(G, m, H)$ is denoted by $\nu(G, m, h)$, $\nu ^\wedge(G, m, h)$ and $\nu^\wedge(G, m, [0, h])$, respectively (see \cite{bajnok2018}). That is,
	
\begin{equation*}
\begin{aligned}
&\nu(G, m, h)	= \max\{|hA| : A \subseteq G, |A|= m\},\\
&\nu^\wedge(G, m, h) = \max\{|h^\wedge A| : A \subseteq G, |A|= m\}, \\
\text{and}\\
&\nu^\wedge(G, m, [0,h]) = \max\{|[0, h]^\wedge A| : A \subseteq G, |A| = m\}. 			
\end{aligned}
\end{equation*}
There is a strong connection between these functions and several other quantities related to Sidon sets and various kinds of sumsets (see, for example, \cite[Chapter C]{bajnok2018}).

For integers $n$ and $r$, let $\binom{n}{r} = \frac{n!}{r! (n-r)!}$ if $0 \leq r \leq n$, and $\binom{n}{r} = 0$ if either $r > n$ or $r < 0$. We have following trivial upper bounds for these functions.
\begin{theorem}[See {\cite[p. 28]{bajnok2018}}]\label{max-sumset-thm1}
Let $m$ and $h$ be positive integers, and let $G$ be an additive abelian group. Then
\begin{enumerate}
  \item \[\nu(G, m, h) \leq \binom{m+h-1}{h}.\]
  \item \[\nu^\wedge(G, m, h) \leq \binom{m}{h}.\]
  \item \[\nu^\wedge(G, m, [0, h]) \leq \sum\limits_{i=0}^{h}\binom{m}{i}.\]
\end{enumerate}
\end{theorem}
	
If $G$ is finite abelian group of order $n$, then following upper bounds are known \cite{bajnok2018}.
	
\begin{theorem}[See {\cite[Proposition A.3]{bajnok2018}}] \label{max-sumset-thm2}
Let $G$ be an finite additive abelian group of order $n$. Then
\[\nu(G, m, h) \leq \min\ \Biggl\{n,\binom{m+h-1}{h}\Biggr\}.\]			
\end{theorem}
		
\begin{theorem}[See {\cite[Proposition A.27]{bajnok2018}}]\label{max-sumset-thm3}
Let $G$ be an finite additive abelian group of order $n$. Then
\[\nu^\wedge(G, m, h) \leq \min \Biggl\{n,\binom{m}{h}\Biggl\}.\]			
\end{theorem}
		
\begin{theorem}[See {\cite[Proposition A.32]{bajnok2018}}]\label{max-sumset-thm4}
Let $G$ be an finite additive abelian group of order $n$. Then
\[\nu^\wedge(G, m, [0, h]) \leq \min \Biggl\{n,\sum\limits_{i=0}^{h}\binom{m}{i}\Biggl\}.\]
\end{theorem}

In case of finite cyclic groups $\mathbb{Z}_n$, the equality holds in Theorem \ref{max-sumset-thm2}, Theorem \ref{max-sumset-thm3} and Theorem \ref{max-sumset-thm4} with many exceptions as observed by Manandhar in \cite{manandhar2012} for some values of $n$. Bajnok posed the following problems:
	
\begin{problem}[{\cite[Problem A.8]{bajnok2018}}]\label{max-sumset-prob1}
Determine all (or, at least, infinitely many) positive integers $n$, $m$ and  $h$, for which 
\[\nu(\mathbb{Z}_n, m, h) < \min \Biggl\{n,\binom{m+h-1}{h}\Biggl\}.\]
\end{problem}
	
\begin{problem}[{\cite[Problem A.31]{bajnok2018}}]\label{max-sumset-prob2}
Determine all (or, at least, infinitely many) positive integers $n$, $m$ and  $h$, for which 
\[\nu^\wedge(\mathbb{Z}_n, m, h) < \min \Biggl\{n,\binom{m}{h}\Biggl\}.\]
\end{problem}
		
\begin{problem}[{\cite[Problem A.34]{bajnok2018}}]\label{max-sumset-prob3}
Determine all (or, at least, infinitely many) positive integers $n$, $m$ and  $h$, for which 
\[\nu^\wedge(\mathbb{Z}_n, m, [0, h]) < \min \Biggl\{n,\sum\limits_{i=0}^{h}\binom{m}{i}\Biggl\}.\]
\end{problem}

In this paper, we prove that for any positive integer $h$, there exist infinitely many positive integers $m$ and $n$ such that the strict inequalities holds in Theorem \ref{max-sumset-thm2}, Theorem \ref{max-sumset-thm3} and Theorem \ref{max-sumset-thm4}. These theorems provide the partial solutions of Problem \ref{max-sumset-prob1}, Problem \ref{max-sumset-prob2} and Problem \ref{max-sumset-prob3} posed by Bajnok. More precisely, we prove the following theorems:
	
\begin{theorem}\label{max-sumset-thm5} 
Let $k$ be an arbitrary positive integer. Let $h\geq2$ be a positive integer. If $m = 2^{\lfloor\log_2 h\rfloor+1}k - 1$ and
$n = \binom{m +  h - 1}{h}$, then 
\[\nu(\mathbb{Z}_n, m, h) < \min\Biggl\{n,\binom{m+h-1}{h}\Biggl\}.\]
\end{theorem}
	
\begin{theorem}\label{max-sumset-thm6} 
Let $k$ be an arbitrary positive integer. Let $h\geq2$ be a positive integer. If $m = 2^{\lfloor\log_2 h\rfloor+1}k + 1$ and $n = \binom{m}{h}$, then 
\begin{equation}\label{max-sumset-thm6-eq1}
  \nu^\wedge(\mathbb{Z}_n, m, h) < \min\Biggl\{n,\binom{m}{h}\Biggl\},
\end{equation} 
and 
\begin{equation}\label{max-sumset-thm6-eq2}
 \nu^\wedge(\mathbb{Z}_n, m, m-h) < \min\Biggl\{n,\binom{m}{m-h}\Biggl\}.
\end{equation} 
\end{theorem}
	
\begin{theorem}\label{max-sumset-thm7} 
Let $k$ be an arbitrary nonnegative integer. Let $h \geq 2$ be a positive integer. If $m =  2^{\lfloor\log_2 h\rfloor}(4 k \ + \ 3)$ and $n = \sum\limits_{i=0}^{h}\binom{m}{i}$, then 
\[\nu^\wedge(\mathbb{Z}_n, m, [0, h]) < \min\Biggl\{n, \sum\limits_{i=0}^{h}\binom{m}{i}\Biggl\}.\]
\end{theorem}

We prove Theorem \ref{max-sumset-thm5}, Theorem \ref{max-sumset-thm6} and Theorem \ref{max-sumset-thm7} in Section \ref{sec2}, Section \ref{sec3} and Section \ref{sec4}, respectively. The proof is based on specific congruence relations satisfied by the sum of elements of a sumset and certain congruence properties of binomial coefficients.

\section{Proof of Theorem \ref{max-sumset-thm5}} \label{sec2}
We begin with the following lemmas.
\begin{lemma}[See {\cite[Section 2.5]{bajnok2018}}]\label{max-sumset-lem1}
Let $m$ and $h$ be positive integers. Then
\begin{equation}\label{max-sumset-lem1-eq1}
  |\Lambda^m(H)| =
  \begin{cases}
    \binom{m+h-1}{h}, & \mbox{if } H = \{h\}~ \text{and}~ \Lambda = [0, h];\\
    \binom{m}{h}, & \mbox{if } H = \{h\}~ \text{and}~ \Lambda = [0, 1]; \\
    \displaystyle\sum_{i=0}^{h}\binom{m}{i}, & \mbox{if } H = [0, h]~ \text{and}~ \Lambda = [0, 1].
  \end{cases}
\end{equation}
\end{lemma}

\begin{lemma}\label{max-sumset-lem2}
Let $m \geq 1$ and $h \geq 2$ be integers, and let $A$ be a subset of $m$ elements of an additive abelian group $G$ such that 
\[|hA| = \binom{m+h-1}{h}.\]
Then
\[|(h - 1)A| = \binom{m + h - 2}{h - 1}.\]
\end{lemma}

\begin{proof}
Let $\Lambda = [0, h]$. Then it follows from Lemma \ref{max-sumset-lem1} that
\begin{equation}\label{max-sumset-lem2-eq1}
|hA| = \binom{m + h - 1}{h} = |\Lambda^m(\{h\})|. 
\end{equation}
Therefore, distinct $m$-tuples in $\Lambda^m(\{h\})$ correspond to distinct elements of $hA$. If 
\[|(h - 1)A| \neq \binom{m + h - 2}{h - 1},\] 
then it follows from Theorem \ref{max-sumset-thm1} and Lemma \ref{max-sumset-lem1} that
\[|(h-1)A| < \binom{m + h - 2}{h-1} = |\Lambda^m(\{h-1\})|.\] 
Therefore, there exist two distinct $m$-tuples $(\lambda_1, \dots,\lambda_m) \in \Lambda^m(\{h-1\})$ and $(\lambda'_1, \dots,\lambda'_m) \in \Lambda^m(\{h-1\})$ such that
 \[\lambda_1 a_1 + \cdots + \lambda_m a_m = \lambda'_1 a_1 + \cdots + \lambda'_m a_m,\]
and so
 \[(\lambda_1 + 1)a_1 + \cdots + \lambda_m a_m = (\lambda'_1 + 1) a_1 + \cdots + \lambda'_m a_m.\]
Thus $(\lambda_1 + 1, \dots,\lambda_m)$ and $(\lambda'_1 +1, \dots,\lambda'_m)$ are distinct $m$-tuples in $\Lambda^m(\{h\})$, which correspond to the same element of $hA$, and so
 \[|hA| < |\Lambda^m(\{h\})| = \binom{m + h - 1}{h}.\] 
This contradicts \eqref{max-sumset-lem2-eq1}. Therefore, we must have
 \[|(h - 1)A| = \binom{m + h - 2}{h-1}.\] 
\end{proof}
	
\begin{lemma} \label{max-sumset-lem4}
Let $h$ and $m$ be positive integers. Let $A = \{a_1, \dots, a_m\}$ be a subset of an additive abelian group $G$ such that
\begin{equation}\label{max-sumset-lem4-eq1}
  |hA| = \binom{m+h-1}{h}.
\end{equation} 
Then
\begin{equation}\label{max-sumset-lem4-eq2} 
\sum_{a \in hA}a = \binom{m+h-1}{h-1}(a_1 + \cdots + a_m).
\end{equation}
\end{lemma}
	
\begin{proof}
The proof is by induction on $h$. If $h=1$, then
\[\sum_{a \in hA} a = \sum_{a \in A} a = (a_1 + a_2 + \dots + a_m) = \binom{m}{0}(a_1 + \dots + a_m).\]
Thus, the lemma is true for $h=1$. Now for the induction hypothesis, assume that the lemma is true for $h-1$ where $h\geq2$. That is, if $|(h-1)A|= \binom{m + h - 2}{h - 1}$, then  
\begin{equation}\label{max-sumset-lem4-eq3} 
\sum_{a \in (h-1)A}a = \binom{m+h-2}{h-2}(a_1 + \dots+ a_m).
\end{equation}
We write
\begin{equation}\label{max-sumset-lem4-eq4}
  \sum_{a \in hA} a = \sum_{i=1}^{m} \alpha_i a_i,
\end{equation}
where $\alpha_1, \ldots, \alpha_m$ are integers to be determined. Values of these integers can be computed as follows: For each $i \in [1, m]$, let
\[S_i = (h-1)A + a_i.\]
The set $S_i$ contains precisely those elements of $hA$ each of which has a representation as a sum of $h$ elements of $A$ with $a_i$ as one of the summands. Thus the coefficient $\alpha_i$ in \eqref{max-sumset-lem4-eq4} is precisely the coefficient of $a_i$ in the sum elements of $S_i$ written as a linear combination of $a_1, \ldots, a_m$. Now, for each $i \in [1, m]$, we have		
\[\sum_{s \in S_i} s = \sum_{a \in (h-1)A + a_i}a = \sum_{a \in (h-1)A}a + |(h-1)A|a_i.\]
Since $|hA| = \binom{m+h-1}{h}$, it follows from Lemma \ref{max-sumset-lem2} that $|(h-1)A| = \binom{m+h-2}{h-1}$. Therefore, using induction hypothesis, it follows from \eqref{max-sumset-lem4-eq3} that		
\[\sum_{s \in S_i} s = \binom{m+h-2}{h-2}(a_1 + a_2 + \cdots+ a_m) + \binom{m+h-2}{h-1}a_i.\]
This implies that
\[\alpha_i = \binom{m+h-2}{h-2} + \binom{m+h-2}{h-1} = \binom{m+h-1}{h-1} \]
Therefore, it follows from \eqref{max-sumset-lem4-eq4} that
\[\sum_{a\in hA}a = \sum_{i=1}^{m}\binom{m+h-1}{h-1}a_i = \binom{m+h-1}{h-1}(a_1 + \cdots + a_m).\]
This completes the proof.
\end{proof}

\begin{proof}[Proof of Theorem \ref{max-sumset-thm5}] 
The proof is by contradiction. If the inequality in Theorem \ref{max-sumset-thm5} is not true, then it follows from Theorem \ref{max-sumset-thm2} that
\[\nu(\mathbb{Z}_n, m, h) = \min\left\{n,\binom{m+h-1}{h}\right\} = n.\]
Therefore, there exists a subset $A$ of $\mathbb{Z}_n$ with $m$ elements, say $A = \{a_1, a_2, \dots, a_m\}$, such that $hA = \mathbb{Z}_n$.
Hence
\begin{equation}\label{max-sumset-thm5-eq1}
\sum_{a \in hA}a \equiv \sum_{a \in \mathbb{Z}_n }a \equiv \frac{n(n+1)}{2} \pmod n.
\end{equation}
Now
\begin{equation*}
\begin{aligned}
n=\binom{m+h-1}{h} &= \frac{(m)(m+1)\dots(m+h-2)(m+h-1)}{h!}\\
&=\frac{(m)(2^{\lfloor\log_2 h\rfloor+1}k) \dots (m+h-1)}{h!}\\
&=\frac{(m)(2^{\lfloor\log_2 h\rfloor+1}k)(c)}{h(h-1)},
\end{aligned}
\end{equation*}
where $c$ is a positive integer such that $c=1$ when $h=2$, and $c=  \frac{(m+2)\dots (m+h-1)}{(h-2)!}$ when $h > 2$. Let $r = \nu_2(h(h-1))$. Since $\gcd(h, h-1) = 1$, it follows that either $2^r$ divides $h$ or $2^r$ divides $h-1$. In any case, $2^r \leq h$, which implies that $r \leq \lfloor\log_2 h \rfloor$. Hence 
\[\nu_2((m)(2^{\lfloor \log_2 h \rfloor+1}k)(c)) > \nu_2(h(h-1)),\]
and so $\nu_2(n) \geq 1$. Thus $n$ is  an even integer. Therefore, it follows from \eqref{max-sumset-thm5-eq1} that	

\begin{equation*}
\sum_{a \in hA}a \equiv \frac{n}{2} \pmod n.
\end{equation*}
Let  $a_1 + \dots+ a_m = x$. Since $n = \binom{m+h-1}{h}$, it follows from Lemma \ref{max-sumset-lem4} that 		
\[\binom{m +h-1}{h-1}x \equiv \frac{1}{2}\binom{m+h-1}{h} \pmod*{\binom{m+h-1}{h}}.\]
Multiplying both sides with $2(h!)$, we get
\[2(h!)\binom{m + h - 1}{h - 1}x \equiv 2(h!)\frac{1}{2}\binom{m+h-1}{h} \pmod*{2(h!)\binom{m + h - 1}{h}}.\]
Simplifying the above expressions, we get		
\[2hx \equiv m  \pmod{2m},\]
which implies that
\[m \equiv 0 \pmod 2.\]
This is a contradiction since $m$ is odd. Therefore, the theorem must be true.
\end{proof}
	
\section{Proof of Theorem \ref{max-sumset-thm6}}\label{sec3}
	
To prove Theorem \ref{max-sumset-thm6} and Theorem \ref{max-sumset-thm7} we need the following lemmas. 

\begin{lemma}\label{max-sumset-lem3}
Let $m$ and $h$ be integers such that $2 \leq h \leq m$, and let $A$ be a subset of $m$ elements of an additive abelian group $G$ such that 
\[|h^\wedge A| = \binom{m}{h}.\]
Let $B = A \setminus \{a\}$ for some $a \in A$. Then
\[|(h - 1)^\wedge B| = \binom{m - 1}{h - 1}.\]
\end{lemma}

\begin{proof}
Let $\Lambda = [0, 1]$. Then it follows from Lemma \ref{max-sumset-lem1} that
\begin{equation}\label{max-sumset-lem3-eq1}
 |h^\wedge A| = \binom{m}{h} = |\Lambda^m(\{h\})|. 
\end{equation}
Therefore, distinct $m$-tuples in $\Lambda^m(\{h\})$ correspond to distinct elements of $h^\wedge A$. Let $A = \{a_1, \ldots, a_m\}$ and $a = a_i$ for some $i \in [1, m]$. Then $B = A \setminus \{a_i\}$. If
\[|(h - 1)^\wedge B| \neq \binom{m - 1}{h - 1},\] 
then it follows from Theorem \ref{max-sumset-thm1} and Lemma \ref{max-sumset-lem1} that
\[|(h - 1)^\wedge B| < \binom{m-1}{h-1} = |\Lambda^{m-1}(\{h-1\})|.\] 
Hence there exist two distinct $m-1$-tuples $(\lambda_1, \ldots, \lambda_{i-1}, \lambda_{i + 1}, \ldots \lambda_m) \in \Lambda^{m-1}(\{h-1\})$ and $(\lambda'_1, \ldots,\lambda'_{i-1}, \lambda'_{i + 1}, \ldots \lambda'_m \in \Lambda^{m-1}(\{h-1\})$ such that
\[\lambda_1 a_1 + \cdots + \lambda_{i-1}a_{i-1} + \lambda_{i+1}a_{i+1} + \cdots + \lambda_m a_m = \lambda'_1 a_1 + \cdots + \lambda'_{i-1}a_{i-1} + \lambda'_{i+1}a_{i+1} + \cdots + \lambda'_m a_m,\]
and so
\[\lambda_1 a_1 + \cdots + \lambda_i a_i + \cdots + \lambda_m a_m = \lambda'_1 a_1 + \cdots + \lambda_i' a_i + \cdots + \lambda'_m a_m,\]
where $\lambda_i = \lambda_i' = 1$. Thus there exist distinct $m$-tuples $(\lambda_1, \dots,\lambda_m)$ and $(\lambda'_1, \dots,\lambda_m')$ in $\Lambda^m(\{h\})$, which correspond to the same element of $h^\wedge A$, and so
\[|h^\wedge A| < |\Lambda^m(\{h\})| = \binom{m}{h}.\] 
This contradicts \eqref{max-sumset-lem3-eq1}. Therefore, we must have
\[|(h - 1)^\wedge B| = \binom{m - 1}{h - 1}.\]
\end{proof}
	
\begin{lemma} \label{max-sumset-lem5}	
Let $m$ and $h$ be positive integers such that $h \leq m$. Let $A = \{a_1, \dots, a_m\}$ be a subset of an additive abelian group $G$ with $m$ elements such that 
\[|h^\wedge A| = \binom{m}{h}.\]
Then 
\[\sum_{a\in h^\wedge A}a =  \binom{m - 1}{h - 1}(a_1 + a_2 + \dots+ a_m).\]
\end{lemma}
	
\begin{proof}
We write
\begin{equation}\label{max-sumset-lem5-eq1}
  \sum_{a \in h^\wedge A} a = \sum_{i=1}^{m} \alpha_i a_i,
\end{equation}
where $\alpha_1, \ldots, \alpha_m$ are integers to be determined. Values of these integers can be computed as follows: For each $i \in [1, m]$, let $B_i = A \setminus \{a_i\}$ and
\[S_i = (h-1)^\wedge B_i + a_i.\]
The set $S_i$ contains precisely those elements of $h^\wedge A$ each of which has a representation as a sum of $h$ distinct elements of $A$ with $a_i$ as one of the summands. Thus the coefficient $\alpha_i$ in \eqref{max-sumset-lem5-eq1} is precisely the coefficient of $a_i$ in the sum elements of $S_i$ written as a linear combination of $a_1, \ldots, a_m$. Therefore, $\alpha_i = |(h-1)^\wedge B_i|$ for each $i \in [1, m]$. Since $|h^\wedge A| = \binom{m}{h}$, it follows from Lemma \ref{max-sumset-lem3} that $\alpha_i = |(h-1)^\wedge B_i| = \binom{m - 1}{h-1}$.
Therefore, it follows from \eqref{max-sumset-lem5-eq1} that
\[\sum_{a\in h^\wedge A}a = \sum_{i=1}^{m}\binom{m-1}{h-1}a_i = \binom{m-1}{h-1}(a_1 + \cdots + a_m).\]
This completes the proof.
\end{proof}
	
\begin{proof}[Proof of Theorem \ref{max-sumset-thm6}] 
The proof is by contradiction. If the inequality \eqref{max-sumset-thm6-eq1} in Theorem \ref{max-sumset-thm6} is not true, then it follows from Theorem \ref{max-sumset-thm3} that
\[\nu(\mathbb{Z}_n, m, h) = \min\left\{n,\binom{m}{h}\right\} = n.\]
Therefore, there exists a subset $A$ of $\mathbb{Z}_n$ with $m$ elements, say $A = \{a_1, a_2, \dots, a_m\}$, such that $h^\wedge A = \mathbb{Z}_n$.
Hence
\begin{equation}\label{max-sumset-thm6-eq3}
\sum_{a \in h^\wedge A}a \equiv \sum_{a \in \mathbb{Z}_n }a \equiv \frac{n(n+1)}{2} \pmod n.
\end{equation}
Now
\begin{equation*}
\begin{aligned}
n = \binom{m}{h} & = \frac{m(m - 1) \dots (m - h + 1)}{h!}\\
& = \frac{m(2^{\lfloor\log_2 h\rfloor + 1})}{h(h - 1)}c,
\end{aligned}
\end{equation*}
where $c$ is a positive integer such that $c = 1$ when $h = 2$, and $c = \frac{(m - 2) \dots (m - h + 1)}{(h - 2)!}$ when $h > 2$. Now an argument similar to the one used in the proof of Theorem \ref{max-sumset-thm5} shows that $n$ is an even integer. Therefore, it follows from \eqref{max-sumset-thm6-eq3} that
\begin{equation*}
\sum_{a\in h^\wedge A}a \equiv \frac{n}{2} \pmod n.
\end{equation*}
Let  $ a_1 + \dots+ a_m = x$. Since $n = \binom{m}{h}$, it follows from Lemma \ref{max-sumset-lem5} that
\[\binom{m-1 }{h-1}x \equiv \frac{1}{2}\binom{m}{h} \pmod*{\binom{m}{h}}.\]
Multiplying both sides with $2(h!)$, we get
\[2(h!)\binom{m-1 }{h-1}x \equiv 2(h!)\frac{1}{2}\binom{m}{h} \pmod*{2(h!)\binom{m}{h}}.\]
Simplifying the above expressions, we get
\[2hx \equiv m \pmod{2m},\]
and so 
\[m \equiv 0 \pmod 2.\]
This is a contradiction since $m$ is odd. Therefore, the inequality \eqref{max-sumset-thm6-eq1} must be true.
		
The inequality \eqref{max-sumset-thm6-eq2} follows from the fact that $\nu^\wedge(\mathbb{Z}_n, m, h) = \nu^\wedge(\mathbb{Z}_n, m, m - h)$ and $\binom{m}{h} = \binom{m}{m - h}$. This completes the proof of the theorem.
\end{proof}

\section{Proof of Theorem \ref{max-sumset-thm7}}\label{sec4}
We begin with the following lemma.	
\begin{lemma} \label{max-sumset-lem6}	
Let $k$ be an arbitrary nonnegative integer, and let $h \geq 2$ be an integer. Let $m = 2^{\lfloor\log_2 h \rfloor}(4k + 3)$. Then
\begin{equation}\label{max-sumset-lem6-eq1}
\begin{aligned}
 \sum_{i = 0}^{h}\binom{m - 1}{i - 1} & \equiv 0 \pmod 2.\\
 \sum_{i = 0}^{h}\binom{m}{i} &\equiv 2 \pmod 4.
\end{aligned}
\end{equation} 
\end{lemma}
	
\begin{proof}
Let $v$ be a positive integer such that $h \in [2^v, 2^ {v + 1} - 1]$. Then$\lfloor\log_2 h\rfloor = v$ and $m = 2^v(4k + 3)$. Let $h = 2^v + r$, where $r\in [0, 2^v - 1]$. We consider two cases. 
		
\noindent\textbf{Case 1} $(r\in \{0, 2, 4, \dots, 2^v-2 \})$.
For each integer $i \in [1, m]$, we have
\[(m - i) \binom{m - 1}{i - 1} = i\binom{m - 1}{i}.\]
Therefore, for each $i \in \{1, 3, \dots, 2^v + r - 1\}$, we have
\[\binom{m - 1}{i - 1} \equiv \binom{m - 1}{i} \pmod 2,\]
and so
\[\binom{m - 1}{i - 1} + \binom{m - 1}{i} \equiv 0  \pmod 2.\]
Hence
\[\sum_{i = 0}^{h}\binom{m - 1}{i - 1} \equiv \sum_{i = 0}^{2^v + r}\binom{m - 1}{i - 1} \equiv 0 \pmod 2.\]

\noindent\textbf{Case 2} $(r\in \{1, 3, \dots, 2^v - 1 \})$. In this case $r - 1 \in \{0, 2, \dots, 2^v - 2 \}$. Therefore, it follows from \textbf{Case 1} that
\[\sum_{i = 0}^{2^v + r}\binom{m - 1}{i - 1} \equiv \sum\limits_{i = 0}^{2^v + r - 1}\binom{m - 1}{i - 1} + \binom{m - 1}{2^v + r - 1} \equiv \binom{m - 1}{2^v + r - 1} \pmod 2.\]
Let  
\[\alpha_r = \binom{m - 1}{2^v + r - 1}.\]
Then
\[\alpha_r = \frac{ (m - 1)(m - 2)\dots(m - (2^v - 1))(m - 2^v) \dots(m - (2^v + r - 1))}{(1)(2) \dots(2^v - 1)(2^v) \dots(2^v + r - 1)}\]
Observe that for each $j \in [1, 2^v + u - 1] \setminus \{2^v\}$, we have
\[\nu_2(m - j) = \nu_2(j),\]
and for $j = 2^v$, we have
\[\nu_2(m - j) = \nu_2(j) + 1.\]
Hence $\nu_2 (\alpha_r) = 1$. This implies that $\alpha_r \equiv 0 \pmod 2$ and $\alpha_r \equiv 2 \ (\textrm{mod} \ 4)$. Hence 
\[\sum\limits_{i = 0}^{h}\binom{m - 1}{i - 1} \equiv 0 \pmod 2.\]
This proves the first congruence relation in \eqref{max-sumset-lem6-eq1}.

Now we prove the second congruence relation in \eqref{max-sumset-lem6-eq1}. By Pascal's identity we have 
\begin{equation*}
\begin{aligned}
\sum\limits_{i = 0}^{h}\binom{m}{i} & =  \sum\limits_{i = 0}^{h}\binom{m - 1}{i - 1} + \sum\limits_{i = 0}^{h}\binom{m - 1}{i}\\
& = \sum\limits_{i = 0}^{h}\binom{m - 1}{i - 1} + \sum\limits_{i = 0}^{h}\binom{m -1}{i - 1} + \binom{m - 1}{h}\\
& = 2\sum\limits_{i = 0}^{h}\binom{m - 1}{i - 1}+\binom{m - 1}{h}.\\
\end{aligned}
\end{equation*}
It follows from the first congruence relation in \eqref{max-sumset-lem6-eq1} that
\[2\sum\limits_{i = 0}^{h}\binom{m - 1}{i - 1} \equiv 0 \ (\textrm{mod} \ 4)\]
Thus 
\[\sum\limits_{i = 0}^{h}\binom{m}{i}\equiv\binom{m - 1}{h} \ (\textrm{mod} \ 4)\] 
Now an argument similar to the one used to prove that $\alpha_r \equiv 2 \pmod 4$ shows that
\[\binom{m - 1}{h} \equiv 2 \pmod 4.\]
Therefore,
\[\sum\limits_{i = 0}^{h}\binom{m}{i} \equiv 2 \ (\textrm{mod} \ 4).\]
This completes the proof of the lemma.
\end{proof}
	
\begin{proof}[Proof of Theorem \ref{max-sumset-thm7}]
The proof is by contradiction. If the inequality in Theorem \ref{max-sumset-thm7} is not true, then it follows from Theorem \ref{max-sumset-thm4} that 
\[\nu^\wedge(\mathbb{Z}_n, m, [0, h]) = \min\Biggl\{n, \sum_{i=0}^{h}\binom{m}{i}\Biggl\} = n.\] 
Therefore, there exists a set $A \subseteq \mathbb{Z}_n$ with $m$ elements, say $A = \{a_1, \ldots, a_m\}$, such that $[0, h]^\wedge A = \mathbb{Z}_n$. Hence
\begin{equation}\label{max-sumset-thm7-eq1}
\sum_{a\in [0, h]^\wedge A}a \equiv \sum_{a\in\mathbb{Z}_n}a \equiv \frac{n(n + 1)}{2} \pmod n.
\end{equation}
Since 
\[|[0, h]^\wedge A| = n = \sum_{i=0}^{h}\binom{m}{i},\]
it follows that $|i^\wedge A| = \binom{m}{i}$ for each $i\in [0, h]$, for if $|i^\wedge A| < \binom{m}{i}$ for some $i$, then 
\[|[0, h]^\wedge A| \leq \sum_{s=0}^{h}|s^\wedge A| < \sum_{i=0}^{h}\binom{m}{i} = n,\]
which is a contradiction. Not also that if $i, j \in [0, h]$ such that $i \neq j$, then $i^\wedge A$ and $j^\wedge A$ are disjoint. Let $x = a_1 + \cdots + a_m$. It follows from Lemma \ref{max-sumset-lem5} that
\[\sum_{a\in i^\wedge A}a =  \binom{m - 1}{i - 1}x.\]
Hence
\begin{equation}\label{max-sumset-thm7-eq2}
 \sum_{a\in [0, h]^\wedge A} a = \sum_{i=0}^{h}\sum_{a\in i^\wedge A}a = \sum\limits_{i = 0}^{h}\binom{m - 1}{i - 1}x.
\end{equation}
If $m = 2^{\lfloor\log_2 h\rfloor}(4k + 3)$, then it follows from Lemma \ref{max-sumset-lem6} that
\[\sum\limits_{i = 0}^{h}\binom{m}{i} \equiv 0 \ (\textrm{mod} \ 2),\] 
and so $n \equiv 0 \pmod 2$. Therefore, it follows from \eqref{max-sumset-thm7-eq1} and \eqref{max-sumset-thm7-eq2} that
\[\sum\limits_{i = 0}^{h}\binom{m - 1}{i - 1}x \equiv \frac{n}{2} \ (\textrm{mod} \ n).\]
Now an application of Lemma \ref{max-sumset-lem6} gives
\[\frac{n}{2}  \equiv 0 \pmod 2,\] 
and so
\[n \equiv 0 \pmod 4.\] 
Therefore,
\[\sum_{i = 0}^{h}\binom{m}{i} \equiv 0 \pmod 4\] 
which contradict the second congruence in Lemma \ref{max-sumset-lem6}. Therefore, the inequality in Theorem \ref{max-sumset-thm7} must be true. This completes the proof.
\end{proof}

% ------------------------------------------------------------------------

\section*{Acknowledgment}
The research of the first named author is supported by the UGC Fellowship (NTA Ref. No.: 231610040283).

% ------------------------------------------------------------------------
\end{document}